\documentclass[runningheads,a4paper]{llncs}

\usepackage{graphicx}
\usepackage{amsmath}
\usepackage{amsfonts}
\usepackage{amssymb}
\usepackage{times}

\usepackage{url}
\urldef{\mailsa}\path|{alfred.hofmann, ursula.barth, ingrid.haas, frank.holzwarth,|
\urldef{\mailsb}\path|anna.kramer, leonie.kunz, christine.reiss, nicole.sator,|
\urldef{\mailsc}\path|erika.siebert-cole, peter.strasser, lncs}@springer.com|

\begin{document}

\mainmatter  

\title{Modules over linear spaces admitting a multiplicative basis}

\titlerunning{Modules over linear spaces admitting a multiplicative basis}

%
%
\author{Antonio J. Calder\'on Mart\'{\i}n%
\thanks{The first and the third authors are supported by the PCI of the UCA `Teor\'\i a de Lie y Teor\'\i a de Espacios de Banach', by the PAI with project numbers FQM298, FQM7156 and by the project of the Spanish Ministerio de Educaci\'on y Ciencia MTM2010-15223. Third author acknowledges the University of Cadiz for the contract research.}
  \and Francisco J. Navarro Izquierdo \and Jos\'e M. S\'anchez Delgado}

\authorrunning{Modules over linear spaces admitting a multiplicative basis}

\institute{Department of Mathematics\\Faculty of Sciences, University of C\'adiz\\
Campus de Puerto Real, 11510, Puerto Real, C\'adiz, Spain.\\
ajesus.calderon@uca.es; javi.navarroiz@uca.es; txema.sanchez@uca.es}

%
%

\toctitle{Modules over linear spaces admitting a multiplicative basis}
\tocauthor{A.J. Calder\'on, F.J. Navarro, J.M. S\'anchez}
\maketitle

\begin{abstract}
We study the structure of certain  modules $V$ over linear spaces
$W$  with restrictions neither on the dimensions nor on the base
field $\mathbb F$. A basis $\mathfrak B = \{v_i\}_{i\in I}$ of $V$
is called multiplicative respect to the basis $\mathfrak B' =
\{w_j\}_{j \in J}$ of $W$ if for any $i \in I, j \in J$ we have
either $v_iw_j = 0$ or $0 \neq v_iw_j \in \mathbb Fv_k$ for some
$k \in I$. We show that if $V$ admits a multiplicative basis then
it decomposes as the direct sum $V=\bigoplus_k  V_k$ of
well-described submodules admitting each one a multiplicative
basis. Also the minimality of $V$ is characterized in terms of the
multiplicative basis and it is shown that  the above direct sum is
by means of the family of its minimal submodules, admitting each
one a multiplicative basis.

\medskip

\end{abstract}

\section{Introduction and previous definitions}

We begin by noting that throughout this paper linear spaces $V$ and $W$ are considered of arbitrary dimensions and over an arbitrary base field $\mathbb F$, and also   the increasing interest in the study of modules over different classes of algebras, and so over linear spaces, specially motivated by their relation with mathematical physics (see \cite{phy1}, \cite{phy2}, \cite{phy3}, \cite{phy4}, \cite{phy5}, \cite{chino1}, \cite{chino2}).
\begin{definition}\rm
Let $V$ be a vector space over an arbitrary base field $\mathbb{F}$. It is said that
$V$ {\it is moduled by a linear space} $W$ (over same base field $\mathbb{F}$), or just that
  $V$ is a {\it  module} over $W$ if it is endowed with a bilinear map
   $V \times W \to V$,  $\hspace{0.2cm}(v, w) \mapsto vw$.
\end{definition}

 Any kind of algebra is an example of a  module over itself. Since the even part $L^0$ of the standard embedding of a Lie triple system $T$ is a Lie algebra, the natural action of $L^0$ over $T$ makes of $T$ a (Lie) module over $L^0$. Hence the present paper extend the results in \cite{Yo_Arb_Alg}.

\begin{definition}\label{11}\rm
Let $V$ be a  module over  the linear space $W$. Given a basis
$\mathfrak{B}'=\{w_j\}_{j\in J}$ of $W$ we say that a basis
$\mathfrak{B}=\{v_i\}_{i\in I}$ of $V$ is {\it multiplicative}
respect $\mathfrak{B}'$ if for any $i \in I$ and $j \in J$ we have
either $v_iw_j=0$ or $0 \neq v_iw_j\in\mathbb{F}v_k$ for some
(unique) $k\in I$.
\end{definition}

To construct examples of modules over linear spaces admitting a
multiplicative basis, we just have to fix two non-empty sets $I$,
$J$ and two arbitrary mappings $\alpha: I \times J \to I$ and
$\beta : I \times J \to \mathbb{F}.$ Then the ${\mathbb F}$-linear
space $V$ with basis  $\mathfrak{B}=\{v_i\}_{i\in I}$ is a module
respect to the ${\mathbb F}$-linear space $W$ with basis
$\mathfrak{B}'=\{w_j\}_{j\in J}$, under the  action induced  by
$v_iw_j := \beta(i,j)v_{\alpha(i,j)}$, admitting   $\mathfrak{B}$
as multiplicative basis respect to $\mathfrak{B}'$.

\begin{remark}\rm
Definition \ref{11} agree with the one  for arbitrary algebras
given in \cite{Yo_Arb_Alg},  and it is a little bit more general
than the usual one in the literature (\cite{m5,m2,m4,Ref,m3}).
\end{remark}

\section{Connections in the set of indexes. Decompositions}

From now on and throughout the paper, $V$ denotes a  module over the linear space $W$, both endowed with respective basis $\mathfrak{B}=\{v_i\}_{i\in I}, \mathfrak{B}'=\{w_j\}_{j \in J}$, and being $\mathfrak{B}$ multiplicative respect to $\mathfrak{B}'$. We denote by $\mathcal P(I)$ the power set of $I$.

\medskip

We begin this section by developing connection techniques among the elements in the set of indexes $I$ as the main tool in our study. For each $j \in J$, a new variable $\overline{j} \notin J$ is introduced and we denote by $\overline{J} := \{\overline j : j \in J\}$ the set of all these new symbols.
We will also write $\overline{(\overline{j})} := j \in J$.
\medskip

We consider the operation $\star: I \times (J\;\dot\cup\;\overline{J}) \to \mathcal{P}(I)$ given by:

\begin{itemize}
\item If $i\in I$ and $j\in J$,
$$\displaystyle i\star j := \left\{
\begin{array}{ccl}
\emptyset &\text{if}& 0=v_iw_j\\
\{k\} &\text{if} &  0\neq v_iw_j\in\mathbb{F}v_k
\end{array}
\right.$$
\item If $i\in I$ and $\overline j\in \overline J$,
$$i\star \overline j:=\{k : 0 \neq v_kw_j \in \mathbb{F}v_i\}$$
\end{itemize}

\medskip

Now, we also consider the mapping $\phi: \mathcal{P}(I)\times (J\;\dot\cup\;\overline{J}) \to \mathcal{P}(I)$ defined as $\phi(U,j) :=\bigcup_{i \in U} (i \star j)$.

\begin{lemma}\label{lema1}
Let  $a,b\in I$ be. Given $j\in J\;\dot\cup\;\overline{J}$ we have
that  $a\in b\star j$ if and only if $b\in a\star \overline{j}$.
\end{lemma}

\begin{proof}
Let us suppose that $a\in b\star j$. If $j\in J$ then $v_bw_j \in {\mathbb F} v_a$, and if $j\in \overline{I}$ we have $ v_aw_{\overline{j}}\in {\mathbb F} v_b$. In any case $b\in a\star \overline{j}$. To prove the converse we can argue in a similar way.
\end{proof}

\begin{lemma}\label{lema2}
Given $j\in J\;\dot\cup\;\overline{J}$ and $U \subset \mathcal{P}(I)$ then $i \in \phi(U,j)$ if and only if $\phi(\{i\},\overline{j})\cap U \neq \emptyset$.
\end{lemma}

\begin{proof}
Let us suppose that $i\in\phi(U,j)$. Then there exists $k\in U$ such that $i\in k\star j$. By Lemma \ref{lema1} we have $k\in i\star\overline{j} = \phi(\{i\},\overline{j})$. So $k\in\phi(\{i\},\overline{j})\cap U\neq \emptyset.$ By arguing in a similar way the converse can be  proven.
\end{proof}

\begin{definition}\label{connection}\rm
Let $i, k \in I$ be  with $i \neq k$. We say that $i $ is {\it connected} to $k$ if there exists a subset $\{j_1,\dots,j_n\}\subset J\;\dot\cup\; \overline{J}$, such that  the following conditions hold:

\begin{enumerate}
\item [{\rm 1.}]
$\phi(\{i\},j_1)\neq\emptyset$, $\phi(\phi(\{i\},j_1),j_{2})\neq\emptyset, \dots, \phi(\phi(\dots\phi(\{i\},j_2)\dots),j_{n-1})\neq\emptyset$.

\medskip

\item [{\rm 2.}]
$k\in\phi(\phi(\dots\phi(\{i\},j_1)\dots),j_{n}).$
\end{enumerate}

We say that $\{j_1,\dots,j_n\}$ is a {\it connection} from $i$ to $k$ and we accept $i$ is connected to itself.
\end{definition}

\begin{lemma}\label{lema3}
Let $\{j_1, j_2, \dots, j_{n-1}, j_n\}$ be any connection from some $i$ to some $k$ where $i,k\in I$ with $i\neq k$. Then the set $\{\overline{j}_n,\overline{j}_{n-1},\dots,\overline{j}_2,\overline{j}_1\}$ is a connection from $k$ to $i$.
\end{lemma}

\begin{proof}
Let us prove it by induction on $n$. For $n=1$ we have that
$k\in\phi(\{i\},j_1)$. It means that $k\in i\star j_1$ and so, by
Lemma \ref{lema1}, $i\in k\star \overline{j}_1 =
\phi(\{k\},\overline{j}_1)$. Hence $\{\overline{j}_1\}$ is a
connection from $k$ to $i$.

Let us suppose that the assertion holds for any connection with $n\geq 1$, elements and let us show this assertion also holds for any connection $\{j_1,j_2,\dots,j_n,j_{n+1}\}.$

By denoting the set
$U:=\phi(\phi(\dots\phi(\{i\},j_1)\dots),j_{n})$ and taking into
the account the second condition of Definition \ref{connection} we
have that $k\in\phi(U,j_{n+1}).$ Then, by Lemma \ref{lema2},
$\phi(\{k\},\overline{j}_{n+1})\cap U\neq\emptyset$ and so we can
take $h\in U$ such that
\begin{equation}\label{eqq1}h\in\phi(\{k\},\overline{j}_{n+1}).\end{equation}
Since $h\in U$ we have that $\{j_1, j_2, \dots, j_{n-1}, j_n\}$ is
a connection from $i$ to $h$. Hence
$\{\overline{j}_n,\overline{j}_{n-1},\dots,\overline{j}_2,\overline{j}_1\}$
connects $h$ with $i$. From here and by Equation (\ref{eqq1}) we
obtain
$i\in\phi(\phi(\dots\phi(\phi(\{k\},\overline{j}_{n+1}),\overline{j}_n)\dots),\overline{j}_{1}).$
So $\{\overline{j}_{n+1},\dots,\overline{j}_2,\overline{j}_1\}$
connects  $k$ with $i$.
\end{proof}

\begin{proposition}
The relation $\sim$ in $I$, defined by $i\sim k$ if and only if $i$ is connected to $k$, is an equivalence relation.
\end{proposition}

\begin{proof}
The reflexive and symmetric character is given by Definition \ref{connection} an Lemma \ref{lema3}.

If we consider the connections $\{a_1,\dots,a_m\}$ and $\{b_1,\dots,b_n\}$ from $a$ to $b$ and from $b$ to $c$ respectively, then is easy to prove that $\{a_1,\dots,a_m,b_1,b_2,\dots,b_n\}$ is a connection from $a$ to $c$. So $\sim$ is transitive and consequently an equivalence relation.
\end{proof}

By the above Proposition we can introduce the quotient set
$I/ \sim := \{[i] : i \in I\},$ becoming $[i]$ the set of elements in $I$ which are connected to $i$.

\medskip

Recall that  a {\it submodule} $Y$ of a  module $V$ (respect to the linear space  $W$) is a linear  subspace of $V$ such that $YW \subset Y$. Our next aim is to associate an (adequate) submodule to each $[i] \in I/\sim$. We define the linear subspace $V_{[i]} :=\bigoplus_{j\in[i]}\mathbb{F}v_j.$

\begin{proposition}\label{lema_submodulo}
For any $i \in I/\sim$ we have that $V_{[i]}$ is a submodule of $V$.
\end{proposition}

\begin{proof}
We need to check $V_{[i]}W \subset V_{[i]}$. Suppose there exist
$i_1 \in [i], j_1 \in [j]$ such that $0 \neq v_{i_1}w_{j_1} \in
v_n$, for some $n \in I$. Therefore $n \in \phi(\{i_1\},j_1)$.
Considering the connection $\{j_1\}$ we get $i_1 \sim n$, and by
transitivity $n \in [i]$.
\end{proof}

\begin{definition}\label{inherited_basis}\rm
We say that a submodule $Y \subset V$ admits a multiplicative basis $\mathfrak{B}_Y$ {\it inherited} from $\mathfrak{B}$ if $\mathfrak{B}_{Y}\subset\mathfrak{B}$.
\end{definition}

Observe that any submodule $V_{[i]} \subset V$ admits  an inherited basis $\mathfrak{B}_{[i]} := \{v_j :j\in[i]\}$. So we can assert

\begin{theorem}\label{theo1}
Let $V$ be a  module admitting a multiplicative basis
$\mathfrak{B}$  respect to a fixed basis of $W$. Then $$V =
\bigoplus_{[i]\in I/\sim}V_{[i]},$$ being any $V_{[i]} \subset V$
a submodule admitting a multiplicative basis $\mathfrak{B}_{[i]}$
inherited from $\mathfrak{B}$.
\end{theorem}

We recall that a  module $V$ is {\it simple} if its only submodules are $\{0\}$ and $V$.

\begin{corollary}\label{coro1}
If $V$ is simple then any couple of  elements of $I$ are connected.
\end{corollary}

\begin{proof}
The simplicity of $V$ applies to get that $V_{[i]} = V$ for some $[i]\in I/\sim$. Hence $[i] = I$ and so any couple of elements in $I$ are connected.
\end{proof}

\section{The minimal components}

In this section we show that, under mild  conditions, the decomposition of $V$ of  Theorem \ref{theo1} can be given by means of the family of its minimal submodules. We begin by introducing a concept of minimality for $V$ that agree with the one for algebras in  \cite{Yo_Arb_Alg}.

\begin{definition}\rm
A  module $V$, (over a linear space $W$), admitting a
multiplicative basis $\mathfrak{B}$ respect to fixed basis of $W$,
is said to be {\it minimal} if its unique nonzero submodule
admitting a multiplicative basis inherited from $\mathfrak{B}$ is
$V$.
\end{definition}

Let us also introduce the concept of $\star$-multiplicativity in the framework of  modules over linear spaces in a similar way to the analogous one for arbitrary algebras  (see \cite{Yo_Arb_Alg} for these notions and examples).

\begin{definition}\rm
We say that a module $V$ respect $W$ admits a $\star${\it -multiplicative basis} $\mathfrak{B} = \{v_i\}_{i\in I}$ respect to a fixed basis  $\mathfrak{B}'=\{l_j\}_{j\in J}$ of $W$, if it is multiplicative and given  $a,b \in I$ such that $b \in a \star j$ for some $j \in J\;\dot\cup\;\overline{J}$ then $v_b \in v_aW$.
\end{definition}

\begin{theorem}\label{theo2}
Let $V$ be a  module respect $W$ admitting a $\star$-multiplicative basis $\mathfrak{B} = \{v_i\}_{i\in I}$ respect to the basis $\mathfrak{B}' = \{w_j\}_{j\in J}$ of $W$. Then $V$ is minimal if and only if
the set of indexes $I$ has all of its elements connected.
\end{theorem}

\begin{proof}
The first implication is similar to Corollary \ref{coro1}. To prove the converse, consider a nonzero submodule $Y \subset V$ admitting a multiplicative basis inherited from $\mathfrak B$. Then, for a certain $\emptyset \neq I_Y \subset I$, we can write $Y =\displaystyle\bigoplus_{i \in I_Y}\mathbb Fv_i$. Fix some $i_0 \in I_Y$ being then
\begin{equation}\label{eqq5}0 \neq v_{i_0}\in Y.\end{equation}

Let us show by induction on $n$ that if $\{j_1,j_2,\dots,j_n\}$ is any connection from $i_0$ to some $k \in I$ then  for any $h \in \phi(\phi(\cdots\phi(\{i_0\},j_1)\dots),j_n)$ we have that $0\neq v_h\in Y$.

In case $n=1$, we get $h\in\phi(\{i_0\},j_1)$. Hence $h\in i_0\star j_1$, then, taking into account that $Y$ is a submodule of $V$, by $\star$-multiplicativity of $\mathfrak B$ and Equation (\ref{eqq5}) we obtain $v_h\in v_{i_0}W \subset Y$.

Suppose now the assertion holds for any connection $\{j_1,j_2,\dots,j_n\}$ from $i_0$ to any $r \in I$ and consider some arbitrary connection $\{j_1, j_2, \dots , j_n, j_{n+1}\}$ from $i_0$ to any $k \in I$. We know that for $x \in U$, where $U := \phi(\phi(\cdots \phi(\{i_0\}, j_1)\cdots ), j_n)$, the element
\begin{equation}\label{eqq6}0\neq v_x \in Y.\end{equation}

Taking into account that the fact $h \in \phi(\phi(\cdots \phi(\{i_0\}, j_1)\dots ), j_{n+1})$ means $h \in \phi(U, j_{n+1}),$ we have that $h \in x \star j_{n+1}$ for some $x \in U$. From here, the $\star$-multiplicativity of $\mathfrak B$ and Equation (\ref{eqq6}) allow us to get $v_h \in v_xW \subset Y$ as desired.

Since given any $k \in I$ we know that $i_0$ is connected to $k$, we can assert by the above observation that $\mathbb Fv_k \subset Y$. We have shown $V=\bigoplus_{k\in I}\mathbb Fv_k \subset Y$ and so $Y = V$.
\end{proof}

\begin{theorem}
Let $V$ be a  module, over the linear space $W$,  admitting a $\star$-multiplicative basis $\mathfrak{B}$ respect to a fixed basis of $W$. Then $V = \bigoplus_k V_k$ is the direct sum of the family of its minimal  submodules, each one admitting a $\star$-multiplicative basis inherited from $\mathfrak B$.
\end{theorem}

\begin{proof}
By Theorem \ref{theo1} we have $V = \displaystyle\bigoplus_{[i]\in
I/\sim} V_{[i]}$ is the direct sum of the submodules $V_{[i]}$.
Now for any $V_{[i]}$ we have that $\mathfrak B_{[i]}$ is a
$\star$-multiplicative basis where all of the elements of $[i]$
are connected. Applying Theorem \ref{theo2} to any $V_{[i]}$ we
have that the decomposition $V=\displaystyle\bigoplus_{[i]\in
I/\sim} V_{[i]}$ satisfies the assertions of the theorem.
\end{proof}

\end{document}